\documentclass[leqno]{amsart}
\usepackage{caption}
\usepackage{amsmath,amssymb,amsthm,mathtools,dsfont,mathrsfs,tikz}
\usepackage{environ}
\usepackage{stmaryrd}
\usetikzlibrary{arrows}
\usetikzlibrary{positioning}
\usetikzlibrary{decorations.text}
\usetikzlibrary{decorations.pathmorphing}
\usetikzlibrary{decorations.pathreplacing}
\makeatletter
\newsavebox{\@brx}
\newcommand{\llangle}[1][]{\savebox{\@brx}{\(\m@th{#1\langle}\)}%
  \mathopen{\copy\@brx\kern-0.5\wd\@brx\usebox{\@brx}}}
\newcommand{\rrangle}[1][]{\savebox{\@brx}{\(\m@th{#1\rangle}\)}%
  \mathclose{\copy\@brx\kern-0.5\wd\@brx\usebox{\@brx}}}
\makeatother
\makeatletter
\newsavebox{\measure@tikzpicture}
\NewEnviron{scaletikzpicturetowidth}[1]{%
  \mathrm{Def}\tikz@width{#1}%
  \mathrm{Def}\tikzscale{1}\begin{lrbox}{\measure@tikzpicture}%
  \BODY
  \end{lrbox}%
  \pgfmathparse{#1/\wd\measure@tikzpicture}%
  \edef\tikzscale{\pgfmathresult}%
  \BODY
}
\makeatother
\usepackage{bbm}
\DeclarePairedDelimiter\norm{\lvert}{\rvert}
\DeclarePairedDelimiter\inner{\langle}{\rangle}
\usepackage{bm}
\usepackage[colorlinks=true,linkcolor=blue,citecolor=blue]{hyperref}
\usepackage{enumitem}
\usepackage{csquotes}

\numberwithin{equation}{section}
\newcounter{intro}

		\newtheorem{introthm}[intro]{Theorem}

		\newtheorem{lem}[equation]{Lemma}
		\newtheorem{cor}[equation]{Corollary}

\theoremstyle{remark}
		\newtheorem{rem}[equation]{Remark}
\theoremstyle{definition}

\title[Characterizations of Nested GVZ-Groups]{Characterizations of Nested GVZ-Groups\\ by central series}
\author{Shawn T. Burkett}
\address{Department of Mathematical Sciences, Kent State University, Kent,
Ohio 44240, U.S.A.} \email{sburket1@kent.edu}
\author{Mark L. Lewis}
\address{Department of Mathematical Sciences, Kent State University, Kent,
Ohio 44240, U.S.A.} \email{lewis@math.kent.edu}
\date{\today}
\begin{document}
\maketitle
\begin{abstract}
Many properties of groups can be defined by the existence of a particular normal series. The classic examples being solvability, supersolvability and nilpotence. Among the nilpotent groups are the so-called nested GVZ-groups --- groups where the centers of the irreducible characters form a chain, and where every irreducible character vanishes off of its center. In this paper, we show that nested GVZ-groups can be characterized by the existence of a certain ascending central series, or by the existence of a certain descending central series. 
\end{abstract}
\section{Introduction}

%\begin{prop}
%Let $\mathsf{S}$ be a supercharacter theory of $G$, let $N$ be $\mathsf{S}$-normal, and let $V={V}(\mathsf{S}\mid N)$. If ${V}(\mathsf{S}\mid V)>V$, then ${V}(\mathsf{S}\mid V)=G$. 
%\end{prop}
%\begin{proof}
% FALSE. Counterexample in \tt{SmallGroup(2\char`\^6,18)}.
%\end{proof}
Throughout this paper, all groups are finite.  We will let ${\rm Irr} (G)$ denote the set of irreducible characters of a group $G$. Given a character $\chi$ of $G$, the {\it vanishing-off subgroup} of $\chi$, denoted ${V}(\chi)$, is the subgroup generated by all elements $g$ of $G$ satisfying $\chi(g)\ne 0$.  This subgroup is introduced by Isaacs in  \cite[Chapter 12]{MI76}.  He calls this the vanishing-off subgroup since it is the smallest subgroup of $G$ so that $\chi$ vanishes off of it.  In \cite{MLvos09}, the second author defines the subgroup ${V}(G)$ to be the subgroup generated by all elements $g$ of $G$ satisfying $\psi(g)\ne0$ for some nonlinear irreducible character $\psi$ of $G$.  In that paper, he studies this subgroup and a number of its properties.  One of the results that he is able to prove is that ${V}(G) = \prod_{\chi \in {\rm Irr}(G)} V(\chi)$.  He also shows that it is the smallest subgroup of $G$ for which every nonlinear irreducible character of $G$ vanishes off of.   

Some of the results regarding ${V}(G)$ are generalized by Mlaiki in \cite{NM14} to case of the irreducible characters of $G$ not containing a fixed normal subgroup $N$ in their kernels.  With this in mind, another variation ${V}(G\mid N)$ of the vanishing-off subgroup is introduced. For a normal subgroup $N$ of $G$, the subgroup ${V}(G\mid N)$ is defined to be the subgroup generated by all elements $g$ of $G$ satisfying $\psi(g)\ne0$ for some irreducible character $\psi$ of $G$ satisfying $N\nleq\ker(\psi)$. This subgroup is the smallest subgroup $V$ so that every irreducible character $\psi$ of $G$ satisfying $N\nleq\ker(\psi)$ also satisfies $\psi(g)=0$ for all $g\in G\setminus V$. These subgroups will be discussed in much more detail in Section~\ref{prelims}. 

In Section~\ref{Galois} of this paper, we define an analog ${U}(G\mid N)$ of ${V}(G\mid N)$. This subgroup arises from a Galois connection on the lattice of normal subgroups of $G$, and is the largest subgroup $U$ of $G$ so that every irreducible character $\psi$ of $G$ satisfying $U\nleq\ker(\psi)$ also satisfies $\psi(g)=0$ for all $g\in G\setminus N$. In particular, when $N={Z}(G)$, the subgroup ${U}(G)={U}(G\mid {Z}(G))$ effectively identifies a set of irreducible characters that are fully ramified over the center; i.e., irreducible characters that vanish off of the center. 

In \cite{MLIGT08}, the second author defines a group $G$ to be a {\it VZ-group}, if every nonlinear irreducible character vanishes off of ${Z}(G)$.  We note that the VZ- here is standing for vanishing off of the center.  Observe that one may characterize a VZ-group $G$ as a group satisfying ${V}(\chi)={Z}(G)$ for all nonlinear irreducible characters $\chi$ of $G$.  Equivalently, $G$ is a VZ-group if and only if ${V}(G) = {Z} (G)$.  We note that VZ-groups had earlier been studied by Kuisch and van der Waall in \cite {KVII} %\bibitem{KVII} R.~W.~van der Waall and E.~B.~Kuisch, Homogeneous character induction II, {\it J. Algebra} {\bf 170} (1994), 584-595
 and by Fern\'{a}ndez-Alcober and Moret\'{o} in \cite{VZ01}.  We can give another characterization of VZ-groups in terms of ${U} (G)$.
 
\begin{introthm}\label{VZ-group}
Let $G$ be a group.  Then $G$ is a VZ-group if and only if ${U}(G)=[G,G]$.
\end{introthm} 

A natural generalization of the definition of a VZ-group is to instead require that ${V}(\chi)={Z}(\chi)$ for all irreducible characters $\chi$ of $G$, where here ${Z}(\chi)=\{g\in G: \norm{\chi(g)} =\chi(1)\}$ is the center of $\chi$. These groups are introduced by Nenciu in \cite{AN12gvz} where she coins the term {\it generalized vanishing center} group, or simply GVZ-group.  

The groups Nenciu studies in \cite{AN12gvz} also satisfy a second hypothesis that generalizes VZ-groups.  Observe that if $G$ is a VZ-group, then every irreducible character $\chi$ of $G$ satisfies either ${Z}(\chi)=G$ or ${Z} (\chi) = {Z} (G)$. In particular, either ${Z}(\chi)\le{Z}(\psi)$ or ${Z} (\psi) \le {Z} (\chi)$ for all characters $\chi, \psi \in \mathrm{Irr}(G)$; i.e., the set  of centers of the irreducible characters, $\{{Z} (\chi) : \chi \in \mathrm{Irr} (G)\}$, forms a chain. A group satisfying this condition is called a {\it nested} group, and when $G$ is a nested group, we say the set of subgroups of $G$ that occur as centers of the irreducible characters is the {\it chain of centers} for $G$ (see \cite{ML19gvz} for details regarding the chain of centers).  It is interesting to note that the study of both nested groups, and GVZ-groups were suggested by Berkovich as Problems \#24 and \#30, respectively, in Research Problems and Themes I of \cite{YBGPPOV1}.

The nested GVZ-groups are the groups that Nenciu studies in \cite{AN12gvz}, and also in the subsequent paper \cite{AN16gvz}. Recently in \cite{ML19gvz}, the second author generalizes several results of Nenciu about nested GVZ-groups to nested groups. Among these results is a description of all irreducible characters of a nested GVZ-group that have a fixed center. This result, appearing as as Lemma 2.6 in \cite{ML19gvz}, allows us to explicitly describe ${U}(G)$ for a nested GVZ-group $G$. It is exactly this result that allows us to use the construction ${U}(G)$ mentioned above to give a characterization of nested GVZ-groups in terms of the existence of certain central series.

One can show that a nested GVZ-group is a $p$-group up to a central direct factor (see \cite[Corollary 2.5]{AN12gvz}).  We generalize this result to any group $G$ satisfying ${U}(G)>1$. 

\begin{introthm}\label{U > 1}
Let $G$ be a group.  If ${U}(G)>1$, then $G=P\times Q$, where $Q$ is an abelian $p'$-group and $P$ is a $p$-group for some prime $p$.  Furthermore, ${U}(G)={U}(P)$.
\end{introthm}

Thus, it suffices to assume that $G$ is a nonabelian $p$-group.  We define in Section~\ref{centralseries} a sequence of subgroups $\upsilon_i$ reminiscent of the upper central series of $G$:  $1 = \upsilon_0 \le \upsilon_1 \le \upsilon_2 \le \dotsb$, where $\upsilon_{i+1}/\upsilon_i = {U} (G/\upsilon_i)$ for each integer $i \ge 0$. The quotients $\upsilon_{i+1}/\upsilon_i$ are central by definition, but its terminal member $\upsilon_\infty$ will not reach $G$ in general. The next Theorem describes exactly when this happens.

\begin{introthm}\label{introthm1}
A nonabelian $p$-group $G$ satisfies $\upsilon_\infty=G$ if and only if $G$ is a nested GVZ-group. In this event, if $G=X_0>X_1>\dotsb>X_n>1$ is the chain of centers for $G$, then $\upsilon_i=[G,X_{n-i}]$ for each $0\le i\le n$, and $\upsilon_{n+1}=G$.\
\end{introthm}

As a consequence of Theorem \ref{introthm1}, we see that nested GVZ $p$-groups  can be defined via the existence of an ascending central series. We also will show that nested GVZ $p$-groups may be defined in terms of a descending central series, which we now describe. We define a chain of subgroups $G = \epsilon_1 \ge \epsilon_2 \ge \dotsb$ by setting $\epsilon_{i+1} = {V} (G \mid [\epsilon_i,G])$ for each integer $i\ge 1$. We also let $\epsilon_\infty$ denote the terminal member of this sequence. Similarly to the above, the quotients $\epsilon_i/\epsilon_{i+1}$ are central in $G/\epsilon_{i+1}$, but this is also not a central series in general. 

\begin{introthm}\label{introthm2}
A nonabelian $p$-group $G$ satisfies $\epsilon_\infty=1$ if and only if $G$ is a nested GVZ-group. In this event, if $G=X_0>X_1>\dotsb>X_n>1$ is the chain of centers for $G$, then $\epsilon_{i+1}=X_i$ for each $0\le i\le n$, and $\epsilon_{n+2}=1$.
\end{introthm}

In particular, whenever $G$ is a nested GVZ-group, these central series recover not only the centers of the irreducible characters of $G$, but also the subgroups $[{Z}(\chi),G]$ for $\chi\in\mathrm{Irr}(G)$. The importance of the latter subgroups to the structure of a nested GVZ-group is illustrated by Lewis in \cite{ML19gvz}.

Observe that if $\upsilon_i$ is a proper subgroup of $[G,G]$, then $\upsilon_{i+1}$ will be contained in $[G,G]$. In particular, since ${U}(G/[G,G])=G/[G,G]$, as $G/[G,G]$ is abelian, we have that $\upsilon_\infty=G$ if and only if $\upsilon_n=[G,G]$ for some $n$. Similarly, $\epsilon_\infty=G$ if and only if $\epsilon_n={Z}(G)$ for some $n$. Therefore, Theorems \ref{introthm1} and \ref{introthm2} can be combined into a single theorem that relates the lengths of these series to the number $\norm{\mathrm{cd}(G)}$ of irreducible character degrees of a nested GVZ-group $G$. This is the main result of the paper.

\begin{introthm}\label{introthm3}
Let $G$ be a $p$-group. The following are equivalent:
\begin{enumerate}[label={\bf(\arabic*)}]
	\item $G$ is a nested GVZ-group and $\norm{\mathrm{cd}(G)}=n$.\\[-2ex]
	\item $\upsilon_{n}=G$ and $\upsilon_{n-1}=[G,G]$.\\[-2ex]
	\item $\epsilon_{n+1}=1$ and $\epsilon_{n}={Z}(G)$.
\end{enumerate}
Moreover, in the event that $G$ is a nested GVZ-group with chain of centers $G=X_0>X_1>\dotsb>X_n>1$, we have $\upsilon_i=[G,X_{n-i}]$ and $\epsilon_{i+1}=X_i$ for each $1\le i\le n$.
\end{introthm}

\section{Preliminaries}\label{prelims}

In this section, we discuss vanishing-off subgroups, nested groups, and GVZ-groups. Before doing so, however, we will fix some notation, which is standard. For a normal subgroup $N$ of $G$, we will identify the irreducible characters of $G/N$ with the irreducible characters of $G$ containing $N$ in their kernels. For this reason, we let $\mathrm{Irr}(G/N)=\{\chi\in\mathrm{Irr}(G):N\le\ker(\chi)\}$. We let $\mathrm{Irr}(G\mid N)$ denote the complement of $\mathrm{Irr}(G/N)$ in $\mathrm{Irr}(G)$. Given a normal subgroup $N$ of $G$, and a character $\vartheta$ of $G$, we let ${I}_G(\vartheta)$ denote the intertia subgroup of $\vartheta$, which is the set of elements of $G$ fixing $\vartheta$ under the natural action of $G$ on $\mathrm{Irr}(N)$. We will denote the $G$-conjugacy class of an element $g\in G$ by $\mathrm{cl}_G(g)$.

Let $\chi\in\mathrm{Irr}(G)$. As in \cite[Chapter 12]{MI76}, we define ${V}(\chi)$ to be the subgroup generated by all elements $g\in G$ for which $\chi(g)\ne0$. Then ${V}(\chi)$ is a normal subgroup of $G$, and is the smallest normal subgroup $V$ of $G$ such that $\chi$ vanishes on $G\setminus V$. Following \cite{NM14}, we define for a normal subgroup $N$ of $G$ the subgroup ${V}(G\mid N)$ to be the subgroup generated by all $g$ such that $\chi(g)\ne 0$ for some $\chi\in\mathrm{Irr}(G\mid N)$. Then ${V}(G\mid N)$ is the smallest subgroup $V$ of $G$ such that every $\chi\in\mathrm{Irr}(G\mid N)$ vanishes on $G\setminus V$.  Also note that if $N=1$, then ${V}(G\mid N)=G$.

\begin{lem}\label{Vprops}
The following statements hold for every pair $H,N$ of normal subgroups of $G$.
\begin{enumerate}[label={\bf(\arabic*)}]
	\item $N\le{V}(G\mid N)$.\\[-2ex]
	\item ${V}(G\mid HN)={V}(G\mid H){V}(G\mid N)$.\\[-2ex]
	\item If $N\le H$, then ${V}(G\mid N)\le{V}(G\mid H)$.\\[-2ex]
	\item If $N\le H$, then ${V}(G/N\mid H/N){V}(G\mid N)/N={V}(G\mid H)/N$.\\[-2ex]
	\item $V(G\mid N)=\prod_{\chi\in\mathrm{Irr}(G\mid N)}V(\chi)$.
\end{enumerate}
\end{lem}

\begin{proof}
To see (1), suppose that there exists $n\in N\setminus{V}(G\mid N)$. Then $\chi(n)=0$ for each $\chi\in\mathrm{Irr}(G\mid N)$, and $n$ lies in the kernel of every other irreducible character of $G$. By column orthogonality (e.g., see Theorem 2.18 of \cite{MI76}), one sees that $\norm{N}=\norm{\mathrm{cl}_G(n)}$, which is strictly less than $\norm{N}$. Thus no such $n$ exists.

We now show (2). To accomplish this, we show that $\mathrm{Irr}(G\mid HN)=\mathrm{Irr}(G\mid H)\cup\mathrm{Irr}(G\mid N)$. Let $\chi\in\mathrm{Irr}(G\mid HN)$ and without loss assume that $H\le \ker(\chi)$. Since $HN\notin\ker(\chi)$, $N$ cannot be contained in $\ker(\chi)$, so $\chi\in\mathrm{Irr}(G\mid N)$. The reverse containment is obvious. Thus, we have
\begin{align*}
{V}(G\mid HN)&=\prod_{\chi\in\mathrm{Irr}(G\mid HN)}{V}(\chi)=\prod_{\substack{\chi\in\mathrm{Irr}(G\mid H)\\\mbox{}\hphantom{bebe}\cup\mathrm{Irr}(G\mid N)}}{V}(\chi)\\
&=\prod_{\chi\in\mathrm{Irr}(G\mid H)}{V}(\chi)\prod_{\chi\in\mathrm{Irr}(G\mid N)}{V}(\chi)\\[2ex]
&={V}(G\mid H){V}(G\mid N).
\end{align*}

Property (3) is immediate from the fact that $\mathrm{Irr}(G\mid N)\subseteq\mathrm{Irr}(G\mid H)$ whenever $N\le H$. 

Next, we show (4). Observe that since $N\le H$, we have $N\le{V}(G\mid H)$, and also that
\[\mathrm{Irr}(G\mid H)=\mathrm{Irr}(G\mid N)\cup\bigl(\mathrm{Irr}(G\mid H)\cap\mathrm{Irr}(G/N)\bigr).\]
It follows that
\[
{V}(G\mid H)={V}(G\mid N)\cdot\prod_{\substack{\chi\in\mathrm{Irr}(G/N)\\H\nleq\ker(\chi)}}{V}(\chi),\]
and the result follows by reducing modulo $N$.

Finally, we show (5). Observe that 
\begin{align*}V(G\mid N)&=\inner{g\in G:\chi(g)\ne 0\ \,\text{for some }\ \,\chi\in\mathrm{Irr}(G\mid N)}\\
				      &=\inner{g\in G:g\in V(\chi)\ \,\text{for some }\ \,\chi\in\mathrm{Irr}(G\mid N)}\\
				      &\le\prod_{\chi\in\mathrm{Irr}(G\mid N)}V(\chi).
\end{align*}
The reverse containment is clear.
\end{proof}

It is an elementary exercise to show that every $\chi\in\mathrm{Irr}(G)$ satisfies the inequality $\chi(1)^2 \le \norm{G:{Z}(\chi)}$, with equality if and only if ${V}(\chi)={Z}(\chi)$ (see \cite[Corollary 2.30]{MI76}. Therefore, one can see that the condition of $G$ being a GVZ-group is equivalent to that of $G$ satisfying $\chi(1)^2 = \norm{G:{Z}(\chi)}$ for all characters $\chi \in \mathrm{Irr}(G)$.  It will at times be advantageous for our purposes to use this description. 

%In \cite{AN12gvz} and \cite{AN16gvz}, these groups were studied by Nenciu, typically under the added assumption that the set $\{{Z}(\chi):\chi\in\mathrm{Irr}(G)\}$ forms a chain. A group satisfying the latter condition is called a nested group. 

%Recently in \cite{ML19gvz}, Lewis studied the structure of nested groups, GVZ-groups, and nested GVZ-groups. In this paper, he generalized several results of Nenciu, and obtained many new results about the character theory of these groups. 

We will list here the results about nested groups from \cite{ML19gvz} that will be needed in the sequel. 

\begin{lem}[{\normalfont\cite[Lemma 2.2, Corollary 2.5, Lemma 2.6]{ML19gvz}}]\label{lewisgvz}
Let $G$ be a nested group with chain of centers $G=X_0>X_1>\dotsb>X_n\ge 1$. The following statements hold. 
\begin{enumerate}[label={\bf(\arabic*)}]
\item $X_n={Z}(G)$.\\[-2ex]
\item $[X_i,G]<[X_{i-1},G]$ for each $1\le i\le n$.\\[-2ex]
\item ${Z}(\chi)=X_i$ for $\chi\in\mathrm{Irr}(G)$ if and only if $[X_i,G]\le\ker(\chi)$ and $[X_{i-1},G]\nleq\ker(\chi)$.
\end{enumerate}
\end{lem}

\section{A Galois connection}\label{Galois}

Let $G$ be a group, and let $N$ be a normal subgroup of $G$. Define the subgroup
\[{U}(G\mid N)=\prod_{\substack{H\lhd G\\{V}(G\mid H)\le N}}H.\]
If $U = {U} (G \mid N)$, then $U$ is the largest normal subgroup of $G$ for which every member of $\mathrm{Irr} (G \mid U)$ vanishes on $G \setminus N$. In particular, if $H \nleq U$, then there exists $\chi \in \mathrm{Irr} (G \mid H)$ that does not vanish off of $N$. Therefore, the subgroup ${U}(G\mid N)$ identifies a set of characters that, in some sense, is maximal with respect to vanishing on $G\setminus N$.

The next result yields the promised Galois connection. For more information on Galois connections, we refer the reader to \cite{galoisprimer}.

\begin{lem}\label{uiff}
Let $H$ and $N$ be normal subgroups of a group $G$. Then $H \le {U} (G\mid N)$ if and only if ${V} (G \mid H) \le N$. In particular the maps $N \mapsto {U} (G \mid N)$ and $N \mapsto {V} (G \mid N)$ give a (monotone) Galois connection from the lattice $\mathrm{Norm} (G)$ of normal subgroups of $G$ to itself. 
\end{lem}

\begin{proof}
If ${V}(G\mid H)\le N$, then it is clear from the definition that $H\le{U}(G\mid N)$.  If $H\le{U}(G\mid N)$, then
\begin{align*}{V}(G\mid H)&\le{V}(G\mid{U}(G\mid N))\\
&=\prod_{\substack{K\lhd G\\{V}(G\mid K)\le N}}{V}(G\mid K)\le N.
\end{align*}
\end{proof}

The following properties hold in general when one has a Galois connection (e.g., see \cite{galoisprimer}). However, we will include a proof for completeness.
\begin{cor}
Let $G$ be a group.  In particular, the following statements hold for each $H,N \lhd G$.
\begin{enumerate}[label={\bf(\arabic*)}]\openup5pt
\item $N\le{U} (G \mid {V} (G \mid N))$.
\item ${V} (G \mid {U} (G \mid N))\le N$.
\item ${U} (G \mid N)\le{U} (G \mid H)$ if $N\le H$.
\item ${U}(G\mid H\cap N)={U}(G \mid H)\cap{U}(G\mid N)$.
\item ${V}(G\mid{U}(G\mid{V} (G \mid N)))={V}(G\mid N)$.
\item ${U}(G\mid{V}(G\mid{U}(G\mid N)))={U}(G\mid N)$.
\item ${V}(G\mid N)=\bigcap\limits_{\substack{H\lhd G\\N\le{U}(G\mid H)}}H$.
\end{enumerate}
\end{cor}
\begin{proof}
Properties (1) and (2) follow immediately from Lemma~\ref{uiff}.

Property (3) is clear from the definition of these subgroups since ${V}(G\mid N)\le{V}(G\mid H)$.

To see (4), first observe that we have ${U}(G\mid H\cap N)\le{U}(G\mid H)\cap{U}(G\mid N)$ by (3). Also we have
\[{V}\bigl(G\mid{U}(G\mid H)\cap{U}(G\mid N)\bigr)\le{V}(G\mid{U}(G\mid H))\cap{V}(G\mid{U}(G\mid N)),\]
by (3) of Lemma~\ref{Vprops}. By (2), the latter subgroup is contained in $H\cap N$, and so statement (4) follows from Lemma~\ref{uiff}.

Now we show (5). By (1), we have $N\le{U}(G\mid{V}(G\mid N))$, and so it follows that ${V}(G\mid N)\le{V}(G\mid{U}(G\mid{V}(G\mid N)))$. Since we also have ${U}(G\mid{V}(G\mid N))\le {U}(G\mid{V}(G\mid N))$, the reverse containment follows from Lemma~\ref{uiff}.

The proof of (6) is similar.

Let $W$ be the intersection of all normal subgroups $H$ satisfying $N\le{U}(G\mid H)$. Since ${V}(G\mid N)\le H$ whenever $N\le{U}(G\mid N)$, we have ${V}(G\mid N)\le W$. By (1), ${V}(G\mid N)$ is among the subgroups $H$ in the intersection defining $W$, so $W\le{V}(G\mid N)$ as well.
\end{proof}

We now present some basic properties of ${U}(G\mid N)$.

\begin{lem}\label{uproperties}
Let $G$ be a nonabelian group.  The following hold. 
\begin{enumerate}[label={\bf(\arabic*)}]\openup5pt
\item For each $N\lhd G$, ${U}(G\mid N)$ is the unique largest subgroup, $U\le G$ such that every character in $\mathrm{Irr}(G\mid U)$ vanishes on $G\setminus N$.
\item For each $N\lhd G$, we have $g\in{U}(G\mid N)$ if and only if every $\chi\in\mathrm{Irr}(G)$ satisfying $g\notin\ker(\chi)$ vanishes on $G\setminus N$.
\item For each $N\lhd G$, we have ${U}(G\mid N)\le N\cap[G,G]$. 
\item If $N$ is characteristic in $G$, so is ${U}(G\mid N)$.
\end{enumerate}
\end{lem}

\begin{proof}
If every character in $\mathrm{Irr}(G\mid H)$ vanishes on $G\setminus N$, then ${V}(G\mid H)\le N$, so $H\le{U}(G\mid N)$ by Lemma~\ref{uiff}. This establishes (1).

To show (2), first note that  for an irreducible character $\chi$, we have $g\in\ker(\chi)$ if and only if $\inner{g}^G\le\ker(\chi)$, where $\inner{g}^G$ denotes the normal closure of $\inner{g}$. Hence  every $\chi\in\mathrm{Irr}(G)$ satisfying $g\notin\ker(\chi)$ vanishes on $G\setminus N$ if and only if every $\chi\in\mathrm{Irr}(G\mid\inner{g}^G)$ vanishes on $G\setminus N$. The latter happens if and only if ${V}(G\mid \inner{g}^G)\le N$, which happens if and only if $\inner{g}^G\le{U}(G\mid N)$. Finally, we note that since ${U}(G\mid N)\lhd G$, ${U}(G\mid N)$ contains $g$ if and only if it contains $\inner{g}^G$.

It is clear from the definition that $U(G\mid N)\le N$. The rest of statement (3) follows from the fact that no linear character can vanish on any element of $G$. In particular, this means that $\mathrm{Irr}(G\mid U(G\mid N))\subseteq\mathrm{Irr}(G\mid[G,G])$.

Finally, we show (4). Let $\tau\in\mathrm{Aut}(G)$, and let $N$ be characteristic in $G$. Let $H\lhd G$ satisfy ${V}(G\mid H)\le N$. To verify (4), it suffices to show that ${V}(G\mid H^\tau)\le N$. To see this, we show that ${V}(G\mid H^\tau)={V}(G\mid H)^\tau$. Since $\chi^\tau(g^\tau)=\chi(g)$, it follows that $V(\chi^\tau)=V(\chi)^\tau$ for every $\chi\in\mathrm{Irr}(G)$. Also since $\ker(\chi^\tau)=\ker(\chi)^\tau$, we have $\mathrm{Irr}(G\mid H)^\tau=\mathrm{Irr}(G\mid H^\tau)$. Hence
\[V(G\mid H^\tau)=\prod_{\chi\in\mathrm{Irr}(G\mid H^\tau)}V(\chi)=\prod_{\chi\in\mathrm{Irr}(G\mid H)}V(\chi)^\tau=V(G\mid H)^\tau.\]
\end{proof}

Whenever $N={Z}(G)$, we simply write ${U}(G)={U}(G\mid N)$. We can now prove Theorem \ref{VZ-group}.

\begin{proof}[Proof of Theorem \ref{VZ-group}]
Suppose first that $G$ is a VZ-group.  Taking $N = {Z}(G)$ in Lemma \ref{uproperties}, we see that $[G,G]$ is the unique largest subgroup of $G$ so that every character in $\mathrm{Irr} (G \mid [G,G])$ vanishes on $G \setminus {Z}(G)$.  Hence, we have ${U} (G) = [G,G]$.  Conversely, if ${U} (G) = [G,G]$, then every irreducible character in $\mathrm{Irr} (G \mid [G,G])$ vanishes on $G \setminus {Z}(G)$.  That is, every nonlinear irreducible character of $G$ vanishes off of ${Z}(G)$, and so, $G$ is a VZ-group.	
\end{proof}

We next consider how ${U}(G)$ interacts with quotients.

\begin{lem}\label{uquotients}
Let $H,N\lhd G$ satisfy ${V}(G\mid N)\le H$. Then 
\[{U}(G/N\mid H/N)={U}(G\mid H)/N.\]
\end{lem}

\begin{proof}
Let $x\mapsto\overline{x}$ denote the canonical surjection $G\to G/N$. Define the sets 
\[\mathcal{C}=\{K\lhd G:{V}(G\mid K)\le H\}\ \,\text{and}\ \,\mathcal{D}= \{\overline{K}\lhd \overline{G}:{V}(\overline{G}\mid \overline{K})\le \overline{H}\}.\] We claim that $\mathcal{D}=\overline{\mathcal{C}}$. However, we first show that $\mathcal{D}=\overline{\mathcal{C}'}$, where
\[\mathcal{C}'=\{K\lhd G:N\le K,\ \,\text{and}\ \,{V}(G\mid K)\le H\}\] To that end, let $K\lhd G$ satisfy $N\le K$ Then $N\le{V}(G\mid K)$, and so we have
\[{V}(\overline{G}\mid \overline{K}){V}(G\mid N)/N={V}(G\mid K)/N.\]
So it follows that $\overline{K}\in\mathcal{D}$ if and only if $K\in\mathcal{C}'$, as claimed. 

In particular, this gives
\[{U}(\overline{G}\mid \overline{H})=\prod_{K\in\mathcal{D}}K=\prod_{K\in\mathcal{C}'}\overline{K}=\overline{\prod_{K\in\mathcal{C}'}K}.\]

Next note that since ${V}(G\mid N)\le H$, we have $K\in\mathcal{C}$ if and only if $KN\in\mathcal{C}'$. Hence
\[{U}(G\mid H)=\prod_{K\in\mathcal{C}}K=\prod_{K\in\mathcal{C}}KN=\prod_{K\in\mathcal{C}'}K.\]
The result now follows by taking quotients.
\end{proof}

We conclude this section by showing that a group $G$ satisfying ${U}(G) > 1$ is essentially a $p$-group in the sense that it is a $p$-group up to a central direct factor. Before doing this however, we review {\it Camina triples}. A triple $(G,M,N)$ where $N\le M$ are normal subgroups of the group $G$ is called a Camina triple if every character  $\chi \in \mathrm{Irr} (M\mid N)$ induces homogeneously to $G$.  These objects were first studied by Mattarei in his Ph.D. thesis \cite{SM92}.  Many more properties of Camina triples can be found in \cite{NM14}, were the following two results can be found.

\begin{lem}\label{NM1}{\normalfont(cf. \cite[Theorem 2.1]{NM14})}\hspace{\labelsep}
	Let $N \le M$ be normal subgroups of the group $G$.  Then the triple $(G,M,N)$ is a Camina triple if and only if ${V}(G\mid N)\le M$.
\end{lem}

\begin{lem}\label{NM2}{\normalfont(cf. \cite[Theorem 2.10]{NM14})}\hspace{\labelsep}
Let $(G, M, N)$ be a Camina triple. If $G/M$ is not a $p$-group for any prime
$p$, then $N\cap {Z}(G) = 1$.
\end{lem}

Observe that, by construction, every $\lambda\in\mathrm{Irr}({Z}(G)\mid {U}(G))$ is fully ramified with respect to $G/{Z}(G)$, and hence induces homogeneously to $G$. In particular, if ${U}(G)>1$ then $(G,{Z}(G),{U}(G))$ is a Camina triple. Note that this yields the next result, which is the first statement of Theorem~\ref{U > 1}.

\begin{lem}\label{nontrivueqpq}
Let $G$ be a nonabelian group. If ${U}(G) > 1$, then $G=P\times Q$, where $p$ is a $p$-group for some prime $p$, and $Q$ is an abelian $p'$-group. In particular, $G$ is nilpotent.
\end{lem}

\begin{proof}
Since ${V}(G\mid{U}(G))\le{Z}(G)$, we have that $(G,{Z}(G),{U}(G))$ is a Camina triple by Lemma~\ref{NM1}. By Lemma~\ref{NM2}, $G/{Z}(G)$ must be a $p$-group since ${U}(G)\le{Z}(G)$. Thus, if $Q$ is a complement for a Sylow $p$-subgroup of ${Z}(G)$, then $Q$ is direct factor of $G$.
\end{proof}

We also obtain a partial converse of Lemma \ref{nontrivueqpq}, which includes the second statement of Theorem~\ref{U > 1}.

\begin{lem}\label{pgps}
Let $G=P\times Q$, where $p$ is a nonabelian $p$-group for some prime $p$, and $Q$ is an abelian $p'$-group. Then
\[{U}(G)={U}(P).\]
\end{lem}

\begin{proof}
Let $H\lhd G$, and note that $H=(P\cap H)\times (Q\cap H)$. Let $\alpha\in\mathrm{Irr}(P)$ and $\beta\in\mathrm{Irr}(Q)$, and let $\chi=\alpha\times\beta$. Then since $(\norm{P},\norm{Q})=1$, we have $\ker(\alpha\times\beta)=\ker(\alpha)\times\ker(\beta)$. So $H\le\ker(\alpha\times\beta)$ if and only if $P\cap H\le\ker(\alpha)$ and $Q\cap H\le\ker(\beta)$. In particular, this means that $\chi\in\mathrm{Irr}(G\mid H)$ if and only if $\alpha\in\mathrm{Irr}(P\mid P\cap H)$ or $\beta\in\mathrm{Irr}(Q\mid Q\cap H)$. 

Assume that $Q\cap H>1$. If $\vartheta\in\mathrm{Irr}(Q\mid Q\cap H)$, then $\mathbbm{1}_P\times\vartheta\in\mathrm{Irr}(G\mid H)$ and $P\le{V}(\mathbbm{1}_P\times\vartheta)\le{V}(G\mid H)$. In particular, ${V}(G\mid H)\ne{Z}(G)$. Therefore, every $H\lhd G$ satisfying ${V}(G\mid H)={Z}(G)$ must also satisfy $H\cap Q=1$. So let $H$ be one such subgroup, and note that $1<H< P$. Then every $\chi\in\mathrm{Irr}(G\mid H)$ has the form $\mu\times\vartheta$ for some $\mu\in\mathrm{Irr}(P\mid H)$ and $\vartheta\in\mathrm{Irr}(Q)$, and such a character satisfies ${V}(\chi)=Q{V}(\vartheta)$. It follows that ${V}(G\mid H)=Q{V}(P\mid H)$, and so ${V}(G\mid H)={Z}(G)$ if and only if ${V}(P\mid H)={Z}(P)$. The result follows.
\end{proof}

This completes the proof of Theorem~\ref{U > 1}, and also yields the following corollary.

\begin{cor}\label{esspgp}
A nonabelian group $G$ satisfies ${U}(G)>1$ if and only if $G/{Z}(G)$ is a $p$-group for some prime $p$, and a Sylow $p$-subgroup of $G$ satisfies ${U}(P)>1$.
\end{cor}

It is clear that the subgroup ${U}(G)$ is a central subgroup of $G$. It turns out that this subgroup is actually contained in the socle of $G$. 

\begin{lem}\label{elabel1}
Let $G$ be a nonabelian $p$-group. Then $U={U}(G)$ is elementary abelian.
\end{lem}

\begin{proof}
If $U=1$, this is clear, so assume that $U>1$. To reach a contradiction, assume that $\exp(U)>p$. Then there exists an irreducible character $\lambda$ of $U$ of order exceeding $p$. Since $\chi(1)=\norm{G:{Z}(G)}^{1/2}$ for every $\chi\in\mathrm{Irr}(G\mid U)$, and $U$ is central, it follows from \cite[Lemma 4.1]{IM01} that $\norm{G:{I}_G(\lambda)}=\norm{G:{Z}(G)}^{1/2}$, which cannot be since $\lambda$ is $G$-invariant and $G$ is nonabelian. Thus no such $\lambda$ exists and it follows that $U$ is elementary abelian.
% that $G/U$ is abelian. In particular, this gives $[G,G]\le U\le{Z}(G)\cap [G,G]$; hence, $U=[G,G]\le{Z}(G)$ and $G$ is a VZ-group. But \cite[Lemma 2.4]{ML09} gives that $[G,G]$ is elementary abelian, which contradicts the assumption that $\exp(U)>p$.
\end{proof}

Combining this result with Lemmas~\ref{nontrivueqpq} and Lemma~\ref{pgps}, we have the following corollary.

\begin{cor}For every nonabelian group $G$, the subgroup ${U}(G)$ is elementary abelian.
\end{cor}

\begin{rem}
In \cite{MLvos09}, the second author showed that the quotient $G/{V}(G)$ is elementary abelian whenever $G$ has a nonabelian nilpotent quotient. The previous result may be considered an analog of this result. In fact, it also shown in this paper that if $G$ is nonabelian and nilpotent, and ${V}(G)<G$, then $G=P\times Q$, where $P$ is a $p$-group and $Q$ is an abelian $p'$-group. In particular, Lemma~\ref{nontrivueqpq} may also be considered an analog of a result appearing in \cite{MLvos09}.
\end{rem}

\section{Central Series and Nested GVZ-Groups}\label{centralseries}

%\begin{defn}
%We call\footnote{This may not be the right definition, but I would like to generalize $\upsilon$ and $\epsilon$ series somehow to give a characterization of nested GVZ-groups in terms of \enquote{vanishing-off series}} a normal series $1=N_0\le N_1\le\dotsb\le N_r=G$ a {\bf vanishing-off} series if 
%\[\mathrm{cd}(G/[N_i,G]\mid[N_{i+1},G]/[N_i,G])=\bigl\{\norm{G/N_i:{Z}(G/N_i)}^{1/2}\bigr\}\] for every $1\le i\le r$.
%\end{defn}

%Idea: Vanishing Series and U Series

%Give the two obvious definitions. Then show how one gives you the other (by vanishing off subgroup or commutators). Vanishing series tests "how nested a GVZ-group is" and U Series tests "how GVZ- a nested group is"

%Note that U(G|N) identifies the characters that "see only N"

%\begin{defn}
%We call a normal series $1=N_0\le N_1\le\dotsb\le N_r=G$ a {\bf relative VZ-} series if every $\mathbf{\chi}\in\mathrm{Irr}(G/N_{i-1}\mid N_i/N_{i-1})$ vanishes off of ${Z}(G/N_{i-1})$.
%\end{defn}

In this section, we will prove the main results of the paper. We first work towards proving Theorem~\ref{introthm1}. Recall that we defined the subgroups $\upsilon_i(G)$ by $\upsilon_0(G)=1$ and $\upsilon_{i+1}(G)/\upsilon_i(G)={U}(G/\upsilon_i(G))$ for each $i\ge 0$. When there is no ambiguity, we will write $\upsilon_i$ instead of $\upsilon_i(G)$. Note that the quotient groups $\upsilon_{i+1}/\upsilon_i$ are central in $G/\upsilon_i$, and are elementary abelian. We wish to determine conditions that guarantee that this sequence of groups terminates in $G$. We recall here that if $\upsilon_{i}< [G,G]$, then $\upsilon_{i+1}\le[G,G]$. We also remark that if $G$ is abelian, then ${U}(G)=G$.

In the event that the $p$-group $G$ is a nested GVZ-group, we obtain an alternative description of ${U} (G)$. In particular, we can show that ${U}(G)>1$ in this case.

\begin{lem}\label{ngvznontrivialu}
Let $G$ be a $p$-group. If $G$ is nested GVZ-group  with chain of centers $G=X_0>X_1>\dotsb>X_n> 1$. Then ${U}(G)=[X_{n-1},G]$.
\end{lem}

\begin{proof}
First note that $X_n={Z}(G)$ by \cite[Lemma 2.2]{ML19gvz}. Also, from \cite[Lemma 2.6]{ML19gvz}, it follows that every $\chi\in\mathrm{Irr}(G\mid [X_{n-1},G])$ has degree $\norm{G:{Z}(G)}^{1/2}$. In particular, every $\chi\in\mathrm{Irr}(G\mid [X_{n-1},G])$ vanishes on $G\setminus{Z}(G)$, and so $[X_{n-1},G]\le U={U}(G)$. 

Now let $\chi\in\mathrm{Irr}(G\mid U)$. Then $\chi$ vanishes on $G\setminus{Z}(G)$, so ${Z}(\chi)={Z}(G)$. By \cite[Lemma 2.6]{ML19gvz}, we have $\chi\in\mathrm{Irr}(G\mid[X_{n-1},G])$ and so it follows that $\mathrm{Irr}(G\mid U)\subseteq\mathrm{Irr}(G\mid[X_{n-1},G])$. This means that $\mathrm{Irr}(G/[X_{n-1},G])\subseteq\mathrm{Irr}(G/U)$ and so 
\[U=\bigcap_{\chi\in\mathrm{Irr}(G/U)}\ker(\chi)\le\bigcap_{\chi\in\mathrm{Irr}(G/[X_{n-1},G])}\ker(\chi)=[X_{n-1},G].\]
\end{proof}

To prove Theorem~\ref{introthm1}, we will induct on $\norm{G}$. The next result is the key to this approach. 

\begin{lem}\label{upsiloninduction}
The group $G$ is a nested GVZ-group if and only if $G/{U}(G)$ is.
\end{lem}

\begin{proof}
If $G$ is a nested GVZ-group, it is clear that every quotient of $G$ is as well. So assume that $G/{U}(G)$ is a nested GVZ-group, and note that we must have $U={U}(G)>1$ by Lemma~\ref{ngvznontrivialu}. Since $G/U$ is a GVZ-group, every $\chi\in\mathrm{Irr}(G/U)$ vanishes on $G\setminus{Z}(\chi)$. We also have that every $\chi\in\mathrm{Irr}(G\mid U)$ vanishes off of $Z(G)$. So $G$ is a GVZ-group. Let $\psi,\chi\in\mathrm{Irr}(G)$, and assume that ${Z}(\psi)\nleq{Z}(\chi)$. Then $\psi\notin\mathrm{Irr}(G\mid U)$, since if it were, we would have ${Z}(\psi)={Z}(G)\le{Z}(\chi)$. So $\psi\in\mathrm{Irr}(G/U)$. If $Z(\chi)=Z(G)$, then $Z(\chi)\le Z(\psi)$. Otherwise, $\chi\in\mathrm{Irr}(G/U)$ and so ${Z}(\chi)/U\le{Z}(\psi)/U$, as $G/U$ is nested. In either case, we have ${Z}(\chi)\le{Z}(\psi)$ and it follows that $G$ is nested.
\end{proof}

We are now ready to prove Theorem~\ref{introthm1}. 

%\begin{thm}\label{char1}A nonabelian $p$-group $G$ satisfies $\upsilon_\infty=G$ if and only if $G$ is a nested GVZ-group. In this event, if $G=X_0>X_1>\dotsb>X_n>1$ is the chain of centers for $G$, then $\upsilon_i=[G,X_{n-i}]$ for each $1\le i\le n$, and $\upsilon_{n+1}=G$.\end{thm}

\begin{proof}[Proof of Theorem~\ref{introthm1}]
We first prove that any group satisfying $\upsilon_\infty(G)=G$ is a nested GVZ-group by induction on $\norm{G}$. Let $U={U}(G)$, which must be nontrivial by Lemma~\ref{ngvznontrivialu}. Then $\upsilon_\infty(G/U)=G/U$, so $G/U$ is a nested GVZ-group by the inductive hypothesis. Hence $G$ is also a nested GVZ-group by Lemma~\ref{upsiloninduction}.  

Now let $G$ be a nested GVZ-group with chain of centers $G=X_0>X_1>\dotsb>X_n>1$. We show that $\upsilon_i=[X_{n-i},G]$ for each $1\le i\le n$ by induction on $i$. By Lemma~\ref{ngvznontrivialu}, we have $\upsilon_1=[X_{n-1},G]$. Let $i\ge 1$ and assume that $\upsilon_i=[X_{n-i},G]$. Then, by Lemma~\ref{lewisgvz}, we have that ${Z}(G/\upsilon_i)=X_{n-i}/\upsilon_i$, and also that $[X_{n-i},G]<[X_{n-i-1},G]$. Therefore, it follows from Lemma~\ref{ngvznontrivialu} that
\[\upsilon_{i+1}/\upsilon_i={U}(G/\upsilon_i)=[X_{n-i-1}/\upsilon_i,G/\upsilon_i]=[X_{n-(i+1)},G]/\upsilon_i,\]
as required. It now follows that $\upsilon_n=[G,G]$, and so $\upsilon_{n+1}=G$.
\end{proof}

%\begin{defn}
%We call a normal series $1=N_0\le N_1\le\dotsb\le N_r=G$ a {\bf vanishing-off} series if ${V}(\chi)\le N_{i-1}$ for all $\chi\in\mathrm{Irr}(G)$ satisfying $N_i\nleq{Z}(\chi)$.
%\end{defn}

We now work towards proving Theorem~\ref{introthm2}. Recall from the introduction that $\epsilon_1=G$, and $\epsilon_{i+1} = {V} (G \mid [\epsilon_{i},G])$ for all integers $i \ge 1$. We first show that $\epsilon_{i+1}\le\epsilon_i$, so that we do in fact have a decreasing chain of subgroups.

\begin{lem}
For each $i\ge 1$, $\epsilon_{i+1}$ is a subgroup of $\epsilon_i$.
\end{lem}
\begin{proof}
We verify this by induction on $i$. If $i=1$, this is clear. So let $i\ge 2$, and assume that $\epsilon_{i+1}\le\epsilon_{i}$. Let $\chi\in\mathrm{Irr}(G\mid[\epsilon_{i+1},G])$; then $\chi\in\mathrm{Irr}(G\mid[\epsilon_{i},G])$ and so $\chi$ vanishes off $V(G\mid[\epsilon_{i},G])=\epsilon_{i+1}$. This forces $V(G\mid[\epsilon_{i+1},G])\le \epsilon_{i+1}$, as required.
\end{proof}

\begin{rem}
It is not true in general that $V(G\mid [N,G])\le N$ for a normal subgroup $N$ of $G$ containing $Z(G)$.
\end{rem}

Observe that $\epsilon_i/\epsilon_{i+1} \le {Z} (G/\epsilon_{i+1})$, for each $i\ge 1$, since  $[\epsilon_i,G] \le {V} (G \mid [\epsilon_i,G]) = \epsilon_{i+1}$. Also observe that if $\epsilon_i > {Z} (G)$ for some $i$, then $[\epsilon_i,G]>1$. Therefore the set $\mathrm{Irr}(G\mid [\epsilon_i,G])$ is nonempty, and so it follows that $\epsilon_{i+1}\ge{Z} (G)$.

\begin{lem}\label{epsiloninduction}
Assume that $\epsilon_\ell(G)={Z}(G)$. If ${V}(G\mid N)\le{Z}(G)$, then $\epsilon_\ell(G/N)\le{Z}(G/N)$.
\end{lem}

\begin{proof}
We claim that $\epsilon_i(G/N)\le\epsilon_i(G)/N$ for each $0\le i\le\ell$ and verify by induction on $i$. This claim is clear if $i=0$, so assume the claim holds for some $0\le i\le\ell-1$. Then 
\begin{align*}
\epsilon_{i+1}(G/N)&={V}(G/N\mid[\epsilon_i(G/N),G/N])\\
			     &\le{V}(G/N\mid[\epsilon_i(G)/N,G/N])\\
			     &={V}(G/N\mid[\epsilon_i(G),G]N/N)\\
			     &={V}(G\mid[\epsilon_i(G),G]){V}(G\mid N)/N\\
			     &\le\epsilon_{i+1}(G){Z}(G)/N\\
			     &=\epsilon_{i+1}(G)/N,
\end{align*}
and the claim is verified. Therefore,
\[\epsilon_\ell(G/N)\le\epsilon_\ell(G)/N={Z}(G)/N\le{Z}(G/N),\]
as desired.
\end{proof}

We may now prove Theorem~\ref{introthm2}.
%This next Theorem contains Theorem \ref{introthm2}.
%
%\begin{thm}
%A nonabelian $p$-group $G$ satisfies $\epsilon_\infty=1$ if and only if $G$ is a nested GVZ-group. In this event, if $G=X_0>X_1>\dotsb>X_n>1$ is the chain of centers for $G$, then $\epsilon_{i+1}=X_i$ for each $0\le i\le n$, and $\epsilon_{n+2}=1$.
%\end{thm}
%
\begin{proof}[Proof of Theorem \ref{introthm2}]
We prove the forward direction by induction on $\norm{G}$. Assume that $\epsilon_\infty=1$, say $\epsilon_{\ell+1}=1$ and $\epsilon_\ell>1$. Then $\epsilon_{\ell}\le{Z}(G)$. It is clear that ${Z}(G)\le\epsilon_{\ell}$, so $X_n={Z}(G)=\epsilon_{\ell}$. This means that every $\chi\in\mathrm{Irr}(G\mid[G,\epsilon_{\ell-1}])$ vanishes on $G\setminus{Z}(G)$, so $1<[\epsilon_{\ell-1},G]\le{U}(G)\le\epsilon_\ell$. By Lemma~\ref{epsiloninduction}, we have that $\epsilon_\infty(G/{U}(G))={U}(G)$, so by the inductive hypothesis, $G/{U}(G)$ is a nested GVZ-group. Hence $G$ is also a nested GVZ-group by Lemma~\ref{upsiloninduction}. 

Now let $G$ be a nested GVZ-group. Then every nonlinear character of $G$ vanishes on $G\setminus X_1$, so $\epsilon_2={V}(G\mid[G,G])\le X_1$. But since $G$ is a GVZ-group and $X_1$ is the center of some $\chi\in\mathrm{Irr}(G)$, we have that $X_1={V}(\chi)\le\epsilon_2$; hence $\epsilon_2=X_1$. So now we may assume that for some $i\ge 1$ that $\epsilon_{i+1}=X_{i}$. Then
\begin{align*}
	{V}(G/[X_{i+1},G]\mid [X_i,G]/[X_{i+1},G])&\le{V}(G\mid [X_i,G])/[X_{i+1},G]\\
		&={V}(G\mid [\epsilon_{i+1},G])/[X_{i+1},G]\\
		&=\epsilon_{i+2}/[X_{i+1},G],
\end{align*}
so we have that every $\chi\in\mathrm{Irr}(G/[X_{i+1},G]\mid [X_i,G]/[X_{i+1},G])$ vanishes on $G\setminus\epsilon_{i+2}$. By Lemma~\ref{lewisgvz}(c), we have that ${Z}(\chi)=X_{i+1}$ for all such $\chi$; it follows that $X_{i+1}\le\epsilon_{i+2}$. Since  $\epsilon_{i+2}$ is the product of some collection of $X_j$s, and $G$ is nested, $\epsilon_{i+2}=X_j$ for some $j$. But we have $X_{i+1}\le\epsilon_{i+2}<\epsilon_{i+1}=X_i$, so it must be that $X_{i+1}=\epsilon_{i+2}$.
\end{proof}

We conclude by proving Theorem \ref{introthm3}.

\begin{proof}[Proof of Theorem~\ref{introthm3}]
If $G$ is a nested GVZ-group with chain of centers $G=X_0>X_1>\dotsb>X_n>1$, then $\mathrm{cd}(G)=\{\norm{G:X_i}^{1/2}:0\le i\le n\}$ and $\norm{\mathrm{cd}(G)}=n+1$. Also, as mentioned earlier, if $n$ is the smallest integer so that $\upsilon_{n+1}=G$, then $n$ must also be the smallest integer so that $\upsilon_n=[G,G]$. So (a) and (b) are equivalent by Theorem~\ref{introthm1}. Similarly, (a) and (c) are equivalent by Theorem~\ref{introthm2}.
\end{proof}

Let $G$ be a nested GVZ-group with chain of centers $G=X_0>X_1>\dotsc>X_n>1$. As a consequence of Theorem~\ref{introthm3}, we obtain that $\upsilon_i=[X_{n-i},G]$ for each $0\le i\le n$. In particular, $[X_{i-1},G]/[X_i,G]$ is an elementary abelian $p$-group for each $1\le i\le n$. By Lemma 4.5 of \cite{ML19gvz}, this actually holds as long as $G$ is nested. Lewis also shows that $X_{i-1}/X_i$ is elementary abelian, and so we have that $\epsilon_i/\epsilon_{i+1}$ is elementary abelian. This means that for a nested GVZ-group $G$, the $\upsilon$-series can be thought of as an ascending exponent-$p$ central series for $[G,G]$, and the $\epsilon$-series gives rise to a descending exponent-$p$ central series for $G/Z(G)$.
\bibliographystyle{plain}
\bibliography{bio}
\end{document}